\documentclass[11pt]{article}
\usepackage{latexsym,amssymb,amsmath,amsthm,enumerate,geometry,float,cite}
\newtheorem{theorem}{Theorem}
\newtheorem{proposition}[theorem]{Proposition}

\usepackage{lineno}
\usepackage{setspace}
\allowdisplaybreaks

\begin{document}
\onehalfspace

\title{Coloring by Pushing Vertices}
\author{Dieter Rautenbach \and Laurin Schwartze \and Florian Werner}
\date{}
\maketitle
\begin{center}
{\small 
Institute of Optimization and Operations Research, Ulm University, Ulm, Germany\\
\texttt{$\{$dieter.rautenbach,laurin.schwartze,florian.werner$\}$@uni-ulm.de}
}
\end{center}

\begin{abstract}
Let $G$ be a graph of order $n$, maximum degree at most $\Delta$, 
and no component of order $2$.
Inspired by the famous 1-2-3-conjecture,
Bensmail, Marcille, and Orenga define a {\it proper pushing scheme} of $G$
as a function $\rho:V(G)\to\mathbb{N}_0$
for which 
$$\sigma:V(G)\to\mathbb{N}_0:u\mapsto 
\left(1+\rho(u)\right)d_G(u)+\sum_{v\in N_G(u)}\rho(v)$$
is a vertex coloring, that is, 
adjacent vertices receive different values under $\sigma$.
They show the existence of a proper pushing scheme $\rho$
with $\max\{ \rho(u):u\in V(G)\}\leq \Delta^2$
and conjecture that this upper bound can be improved to $\Delta$.
We show their conjecture for cubic graphs and regular bipartite graphs.
Furthermore, we show the existence of a proper pushing scheme $\rho$
with $\sum_{u\in V(G)}\rho(u)\leq \left(2\Delta^2+\Delta\right)n/6$.\\[3mm]
\end{abstract}

\section{Introduction}

We consider finite, simple, and undirected graphs and use standard terminology.
For a graph $G$, let $V(G)$ denote the vertex set of $G$
and let $E(G)$ denote the edge set of $G$.
For a vertex $u$ of $G$, let $N_G(u)$ denote the neighborhood of $u$ in $G$
and let $d_G(u)$ denote the degree of $u$ in $G$. 
Let $n(G)$, $m(G)$, and $\Delta(G)$ denote 
the order $|V(G)|$,
the size $|E(G)|$, and
the maximum degree $\max\{ d_G(u):u\in V(G)\}$ of $G$, respectively.
A graph is {\it nice} if it has no component of order two.
Let $\mathbb{N}$ be the set of positive integers and 
let $\mathbb{N}_0=\mathbb{N}\cup \{ 0\}$.
For a non-negative integer $k$, 
let $[k]$ be the set of all positive integers at most $k$ 
and let $[k]_0=\{0\}\cup [k]$.

Inspired by the famous 1-2-3-conjecture \cite{be,se},
posed in 2004 by Karo\'nski, \L uczak, and Thomason \cite{kaluto} and 
solved in 2024 by Keusch \cite{ke},
as well as by its wider context 
combining graph labeling and irregularity \cite{chjaleoerusa,ga},
Bensmail, Marcille, and Orenga \cite{bemaor} 
recently proposed coloring graphs by so-called pushing schemes.
The 1-2-3-conjecture states that the edges of every nice graph $G$
can be labeled with $1$, $2$, and $3$ in such a way that,
for every two adjacent vertices, 
the sums of the labels of the incident edges are distinct,
that is, there is a labeling $\ell:E(G)\to [3]$ such that 
\begin{eqnarray}\label{e-2}
\sum\limits_{w\in N_G(u)}\ell(uw)\not=\sum\limits_{w\in N_G(v)}\ell(vw)
\mbox{ for every edge $uv$ of $G$}.
\end{eqnarray}
Now, Bensmail, Marcille, and Orenga consider a setting 
where the label $\ell(e)$ of each edge $e=uv$
arises by {\it pushing} the incident vertices $u$ and $v$ a certain number of times.
More precisely, {\it pushing} $\rho(u)$ times the vertex $u$ 
and $\rho(v)$ times the vertex $v$ 
results in the edge label $\ell(e)=1+\rho(u)+\rho(v)$.
They call a function 
$$\rho:V(G)\to\mathbb{N}_0$$ 
assigning a {\it pushing value} $\rho(u)$ to every vertex $u$ of $G$
a {\it proper pushing scheme} of $G$
if (\ref{e-2}) holds for $\ell(e)=1+\rho(u)+\rho(v)$
or, equivalently, if
$$\sigma(u)\not=\sigma(v)
\mbox{ for every edge $uv$ of $G$,}$$
where 
\begin{eqnarray}\label{e-1}
\sigma:V(G)\to\mathbb{N}_0:u\mapsto\left(1+\rho(u)\right)d_G(u)+\sum\limits_{v\in N_G(u)}\rho(v),
\end{eqnarray}
that is, the function $\sigma:V(G)\to\mathbb{N}_0$ 
derived from $\rho$ is a vertex coloring of $G$.
Unless stated otherwise, 
whenever we consider functions $\rho$ and $\sigma$ 
defined on the vertex set of some graph,
they are related as in (\ref{e-1}). 

For a nice graph $G$, 
Bensmail et al.~\cite{bemaor} define
\begin{eqnarray*}
P^1(G)&=& \min\left\{\max\limits_{u\in V(G)}\rho(u):\mbox{ $\rho$ is a proper pushing scheme of $G$}\right\}\mbox{ and }\\
P^t(G)&=& \min\left\{\sum\limits_{u\in V(G)}\rho(u):\mbox{ $\rho$ is a proper pushing scheme of $G$}\right\}.
\end{eqnarray*}
The definitions immediately imply
\begin{eqnarray}\label{e0}
P^1(G)\leq P^t(G)\leq P^1(G)n(G).
\end{eqnarray}
Bensmail et al.~\cite{bemaor} show 
\begin{eqnarray}\label{e1}
P^1(G)\leq \Delta(G)^2
\end{eqnarray}
and conjecture  
\begin{eqnarray}\label{e2}
P^1(G)\leq \Delta(G).
\end{eqnarray}
In view of (\ref{e0}), the inequality (\ref{e1}) implies 
\begin{eqnarray}\label{e3}
P^t(G)\leq \Delta(G)^2n(G)
\end{eqnarray}
and the conjecture (\ref{e2}) motivates the weaker conjecture 
\begin{eqnarray}\label{e4}
P^t(G)\leq \Delta(G)n(G).
\end{eqnarray}
Bensmail et al.~\cite{bemaor} determine $P^1$
for complete graphs, 
complete bipartite graphs,
paths,
trees, 
cycles, and
cacti.
The also show that deciding, for a given nice graph $G$, whether $P^1(G)\leq 1$
is NP-complete, and they also establish the hardness of $P^t$.

\pagebreak

Postponing all proofs and some statements to the next section,
we now discuss our main contributions.
Our first main result verifies conjecture (\ref{e2}) for cubic graphs.

\begin{theorem}\label{theorem0}
If $G$ is a cubic graph, then $P^1(G)\leq 3$.
\end{theorem}
We also establish conjecture (\ref{e2}) for regular bipartite graphs,
see Proposition \ref{proposition0} below.

The proof of (\ref{e1}) in \cite{bemaor} relies on a natural greedy algorithm,
which we explain in the next section.
Our second main result improves (\ref{e3}) 
and is based on the greedy algorithm and the probabilistic first-moment method.

\begin{theorem}\label{theorem1}
If $G$ is a nice graph of order $n$ and maximum degree at most $\Delta$,
then the greedy algorithm applied to some linear ordering of the vertices of $G$ 
yields a proper pushing scheme $\rho:V(G)\to\mathbb{N}_0$ with 
$$P^t(G)\leq \sum\limits_{u\in V(G)}\rho(u)\leq \sum\limits_{u\in V(G)}\frac{1}{6}d_G(u)(2d_G(u)+1)\leq \frac{n\Delta\left(2\Delta+1\right)}{6}.$$
\end{theorem}
Our computational experiments indicate that the greedy algorithm 
is much better 
than expressed by (\ref{e1}) or Theorem \ref{theorem1}
at least on average.
In fact, executing the algorithm on 
all $4$-regular graphs of order at most $15$ and 
considering few random orderings for each of these graphs,
the greedy algorithm is strong enough to establish conjecture 
(\ref{e2}) for these graphs.
The next section contains two further minor results
both pointing to possible improvements of 
(the analysis of) the greedy algorithm.

\section{Proofs and further results}

We immediately proceed to the:

\begin{proof}[Proof of Theorem \ref{theorem0}]
Clearly, we may assume that $G$ is connected.
If $G=K_4$, then $P^1(G)=3$ (cf. Observation 3.2 in \cite{bemaor}).
Hence, we may assume that $G\not=K_4$, and 
Brooks' Theorem implies that $G$ has chromatic number at most $3$.
Consider a three coloring of $G$ with (possibly empty) color classes $X$, $Y$, and $Z$,
where the coloring is chosen in such a way that $|Z|$ is as large as possible and,
subject to this condition, $|Y|$ is as large as possible.
It follows that every vertex in $X$ has a neighbor in $Y$ and a neighbor in $Z$,
and that every vertex in $Y$ has a neighbor in $Z$.
Let $Y'$ be the set of vertices in $Y$ that have a neighbor in $X$,
and let $Y''=Y\setminus Y'$.
Note that the vertices in $Y''$ have all their neighbors in $Z$.
Let $Z'$ be the set of vertices in $Z$ that have a neighbor in $X$ or $Y'$,
and let $Z''=Z\setminus Z'$.
Note that the vertices in $Z''$ have all their neighbors in $Y''$.
Let $Y_0$ be the set of vertices in $Y''$ that have no neighbor in $Z''$,
and let $Y'''=Y''\setminus Y_0$.
Altogether,
\begin{itemize}
\item every vertex in $X$ has a neighbor in $Y'$ and a neighbor in $Z'$
but no neighbor in $Y''\cup Z''$,
\item every vertex in $Y'$ has a neighbor in $X$ and a neighbor in $Z'$
but no neighbor in $Z''$,
\item every vertex in $Y_0$ has all its neighbors in $Z'$,
\item every vertex in $Y'''$ has all its neighbors in $Z$ and at least one neighbor in $Z''$,
\item every vertex in $Z'$ has a neighbor in $X$ or $Y'$, and
\item every vertex in $Z''$ has all its neighbors in $Y'''$.
\end{itemize}
For $i\in \mathbb{N}$, let $D_i=\left\{ u\in Y'''\cup Z'':{\rm dist}_G(u,Z')=i\right\}$.
Note that 
$\bigcup\limits_{i\in \mathbb{N}} D_{2i-1}=Y'''$,
$\bigcup\limits_{i\in \mathbb{N}} D_{2i}=Z''$, 
every vertex in $D_1$ has a neighbor in $D_2$, and
every vertex in $D_i$ for some $i\geq 2$ has a neighbor in $D_{i-1}$.

In order to complete the proof,
we show that 
$$\rho:V(G)\to [3]_0:
u\mapsto
\begin{cases}
0 & \text{, if }u\in X\cup \bigcup\limits_{i\in \mathbb{N}} D_{3i-2},\\
1 & \text{, if }u\in Y',\\[2mm]
2 & \text{, if }u\in \bigcup\limits_{i\in \mathbb{N}} D_{3i-1},\mbox{ and}\\
3 & \text{, if }u\in Y_0\cup Z'\cup \bigcup\limits_{i\in \mathbb{N}} D_{3i}
\end{cases}
$$
is a proper pushing scheme.

See Figure \ref{fig0} for an illustration.

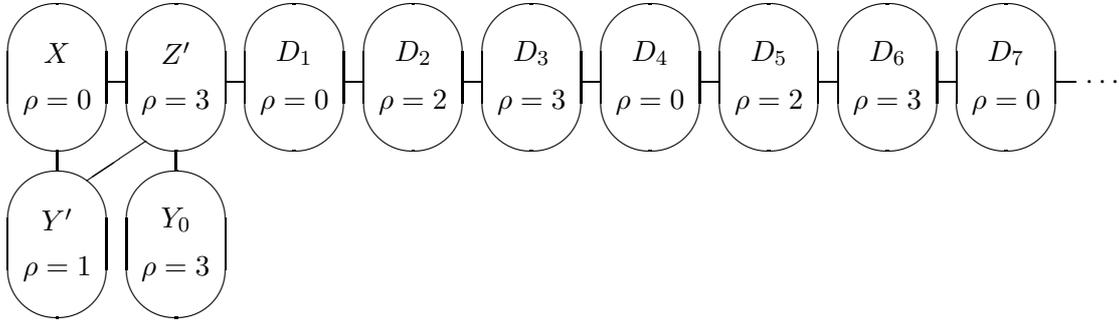
\begin{figure}[H]
\begin{center}
\unitlength 1.3mm 
\linethickness{0.4pt}
\ifx\plotpoint\undefined\newsavebox{\plotpoint}\fi 
\begin{picture}(111,32)(0,0)
\put(5,22){\makebox(0,0)[cc]{$\rho=0$}}
\put(5,5){\makebox(0,0)[cc]{$\rho=1$}}
\put(17,22){\makebox(0,0)[cc]{$\rho=3$}}
\put(17,5){\makebox(0,0)[cc]{$\rho=3$}}
\put(29,22){\makebox(0,0)[cc]{$\rho=0$}}
\put(41,22){\makebox(0,0)[cc]{$\rho=2$}}
\put(53,22){\makebox(0,0)[cc]{$\rho=3$}}
\put(89,22){\makebox(0,0)[cc]{$\rho=3$}}
\put(65,22){\makebox(0,0)[cc]{$\rho=0$}}
\put(101,22){\makebox(0,0)[cc]{$\rho=0$}}
\put(77,22){\makebox(0,0)[cc]{$\rho=2$}}
\put(5,27){\makebox(0,0)[cc]{$X$}}
\put(5,10){\makebox(0,0)[cc]{$Y'$}}
\put(17,27){\makebox(0,0)[cc]{$Z'$}}
\put(17,10){\makebox(0,0)[cc]{$Y_0$}}
\put(29,27){\makebox(0,0)[cc]{$D_1$}}
\put(41,27){\makebox(0,0)[cc]{$D_2$}}
\put(53,27){\makebox(0,0)[cc]{$D_3$}}
\put(89,27){\makebox(0,0)[cc]{$D_6$}}
\put(65,27){\makebox(0,0)[cc]{$D_4$}}
\put(101,27){\makebox(0,0)[cc]{$D_7$}}
\put(77,27){\makebox(0,0)[cc]{$D_5$}}
\put(5,24.5){\oval(10,15)[]}
\put(5,7.5){\oval(10,15)[]}
\put(17,24.5){\oval(10,15)[]}
\put(17,7.5){\oval(10,15)[]}
\put(29,24.5){\oval(10,15)[]}
\put(41,24.5){\oval(10,15)[]}
\put(53,24.5){\oval(10,15)[]}
\put(89,24.5){\oval(10,15)[]}
\put(65,24.5){\oval(10,15)[]}
\put(101,24.5){\oval(10,15)[]}
\put(77,24.5){\oval(10,15)[]}
\put(10,24){\line(1,0){2}}
\put(22,24){\line(1,0){2}}
\put(34,24){\line(1,0){2}}
\put(46,24){\line(1,0){2}}
\put(58,24){\line(1,0){2}}
\put(94,24){\line(1,0){2}}
\put(70,24){\line(1,0){2}}
\put(106,24){\line(1,0){2}}
\put(82,24){\line(1,0){2}}
\put(17,17){\line(0,-1){2}}
\put(5,17){\line(0,-1){2}}
\put(14,18){\line(-3,-2){6}}
\put(111,24){\makebox(0,0)[cc]{$\cdots$}}
\end{picture}
\end{center}
\caption{The partition of $V(G)$, the values of $\rho$, and the possible edges.}\label{fig0}
\end{figure}
In fact, the function $\sigma$ as in (\ref{e-1}) satisfies
$$\sigma(u)\in 
\begin{cases}
\{ 8,10\} & \text{, if }u\in X,\\
\{ 9,12\} & \text{, if }u\in Y',\\
\{ 21\} & \text{, if }u\in Y_0,\\
\{ 12,13,\ldots,19\} & \text{, if }u\in Z',\\
\{ 10,11\} & \text{, if }u\in D_1,\\
\{ 10,11,12\} & \text{, if }u\in  \bigcup\limits_{i\in \mathbb{N}} D_{3i+1},\\
\{ 9,12,15\} & \text{, if }u\in  \bigcup\limits_{i\in \mathbb{N}} D_{3i-1},\text{ and}\\
\{ 14,16,18\} & \text{, if }u\in  \bigcup\limits_{i\in \mathbb{N}} D_{3i}.
\end{cases}
$$
Furthermore, 
if $\sigma(u)=12$ for some vertex $u\in Z'$, then $u$ has no neighbor in $Y'$, and,
if $\sigma(u)=12$ for some vertex $u\in \bigcup\limits_{i\in \mathbb{N}} D_{3i+1}$,
then $u$ has all its neighbors in $\bigcup\limits_{i\in \mathbb{N}} D_{3i}$.
It follows that $\sigma$ is a vertex coloring of $G$, 
which completes the proof.
\end{proof}
The above proof strongly exploits that the considered graphs are regular.
A simple variation of this proof yields conjecture (\ref{e2}) for regular bipartite graphs. 
It seems much harder to show the conjecture for subcubic graphs.

\begin{proposition}\label{proposition0}
If $G$ is a $\Delta$-regular bipartite graph for some $\Delta\geq 4$, 
then $P^1(G)\leq \Delta$.
\end{proposition}
\begin{proof}
Clearly, we may assume that $G$ is connected.
Let $r$ be any vertex of $G$.
Let $D_i=\{ u\in V(G):{\rm dist}_G(u,r)=i\}$ for $i\in \mathbb{N}_0$.
Since each of the sets $D_i$ is contained in one of the two partite sets of $G$,
these sets are independent.
Furthermore, if $|i-j|\geq 2$, then there is no edge between $D_i$ and $D_j$.
Hence, the sets
$C_j=\{u\in V(G):{\rm dist}_G(u,r)\equiv j\mod 3\}$ for $j\in \{ 0,1,2\}$
are independent.

In order to complete the proof,
we show that 
$$\rho:V(G)\to [\Delta]_0:
u\mapsto
\begin{cases}
\Delta & \text{, if }u\in C_0,\\[2mm]
1 & \text{, if }u\in C_1,\mbox{ and}\\
0 & \text{, if }u\in C_2
\end{cases}
$$
is a proper pushing scheme.

Let $u$ be a vertex such that ${\rm dist}_G(u,r)\geq 1$.
Let $u\in C_k$.
Since
all neighbors of $u$ are in $C_{(k-1)\mod 3}\cup C_{(k+1)\mod 3}$
and $u$ has at least one neighbor in $C_{(k-1)\mod 3}$,
we obtain
$$
\sigma(u)\in
\begin{cases}
\big\{\Delta^2+\Delta+k_0:k_0\in [\Delta-1]_0\big\} & \text{, if }u\in C_0,\\[2mm]
\big\{2\Delta+k_1\Delta:k_1\in [\Delta]\big\} & \text{, if }u\in C_1,\mbox{ and}\\
\big\{\Delta+k_2+\Delta(\Delta-k_2):k_2\in [\Delta]\big\} & \text{, if }u\in C_2.
\end{cases}
$$
Note that $r\in C_0$ and $\sigma(r)=\Delta^2+
2\Delta$.
Furthermore, if $u$ is a neighbor of $r$, 
then $u\in C_1$ and $\sigma(u)=3\Delta\not=\sigma(r)$.
Now, let $uv$ be an edge of $G$ that is not incident with $r$.
By symmetry, we may assume that $u\in C_i$ and $v\in C_{(i+1)\mod 3}$.
If $i=0$, then 
$\sigma(u)\not\equiv 0\mod \Delta$,
because $u\in C_0$ has $k_0\geq 1$ neighbors in $C_1$, and 
$\sigma(v)\equiv 0\mod \Delta$.
If $i=1$, then $2\Delta<\sigma(u)\equiv 0\mod \Delta$
and either $\sigma(v)\not\equiv 0\mod \Delta$
(for $k_2\not=\Delta$)
or $\sigma(v)=2\Delta$ (for $k_2=\Delta$).
If $i=2$, then 
$\sigma(u)\leq \Delta^2+1$ 
and 
$\sigma(v)\geq \Delta^2+\Delta$.
In all three cases, we obtain $\sigma(u)\not=\sigma(v)$,
which completes the proof.
\end{proof}
For the proof of Theorem \ref{theorem1},
we explain the greedy algorithm from \cite{bemaor}.
In our exposition 
we slightly deviate from \cite{bemaor} 
in the way we handle vertices of degree $1$;
this deviation is not essential yet simpler:
Let $G$ be a nice graph of order $n$ and maximum degree at most $\Delta$.
Let $u_1,\ldots,u_n$ be a linear ordering of the vertices of $G$.
The greedy algorithm deterministically produces a sequence 
$\rho_0,\rho_1,\ldots,\rho_n$ of functions $\rho_i:V(G)\to\mathbb{N}_0$
such that, for every $i\in [k]_0$, the following two properties hold:
\begin{itemize}
\item Property $P_i^{(1)}$: $\rho_i(u_j)=0$ for every $j\in [n]\setminus [i]$.
\item Property $P_i^{(2)}$: For 
$$\sigma_i:V(G)\to\mathbb{N}_0:
u\mapsto \left(1+\rho_i(u)\right)d_G(u)+\sum\limits_{v\in N_G(u)}\rho_i(v)$$
there is no edge $u_ju_k$ with $j,k\in [i]$ and $\sigma_i(u_j)=\sigma_i(u_k)$.
\end{itemize}
By $P_0^{(1)}$, we have $\rho_0(u)=0$ for every vertex $u$ of $G$ 
and the condition $P_0^{(2)}$ is void.
By $P_n^{(2)}$, we have that $\rho_n$ is a proper pushing scheme.

Now, suppose that, for some $i\in [n]$, the function $\rho_{i-1}$ satisfies $P_{i-1}^{(1)}$ and $P_{i-1}^{(2)}$.
Note that, in particular, we have $\rho_{i-1}(u_i)=0$.
Let 
$$\rho_i(u_j)=
\begin{cases}
\rho_{i-1}(u_j)&\mbox{ for $j\in [n]\setminus \{ i\}$, and}\\
s &\mbox{ for $j=i$},
\end{cases}$$
where $s$ is the smallest non-negative integer such that the function $\rho_i$ satisfies $P_i^{(2)}$; note that $P_i^{(1)}$ holds by construction.
The key observation of Bensmail et al.~ \cite{bemaor} 
is that a valid choice for $s$ exists and that, in fact, it satisfies $s\leq \Delta^2$.
If $u_i$ has degree at most $1$, then $s=0$ is a valid choice, because $G$ is nice, 
that is, the functions $\rho_i$ and $\rho_{i-1}$ are the same.
If $u_i$ has degree at least $2$, then 
\begin{itemize}
\item for every $j\in [i-1]$ with $u_iu_j\in E(G)$, 
the condition ``$\sigma_i(u_i)\not=\sigma_i(u_j)$'' 
excludes at most one non-negative integer as a valid choice for $s$, and
\item for every two distinct $j,k\in [i-1]$ with $u_iu_j,u_ju_k\in E(G)$ and $u_iu_k\not\in E(G)$, 
the condition ``$\sigma_i(u_j)\not=\sigma_i(u_k)$'' 
excludes at most one non-negative integer as a valid choice for $s$.
\end{itemize}
Since all other relevant edges are not affected by changing the value at $u_i$,
this implies that 
\begin{eqnarray}\label{e5}
\rho_n(u_i)=\rho_i(u_i)\in \left[s_i^{(1)}+s_i^{(2)}\right]_0,
\end{eqnarray}
where 
\begin{itemize}
\item $s_i^{(1)}$ is the number of $j$ with $j\in [i-1]$ and $u_iu_j\in E(G)$, and
\item $s_i^{(2)}$ is the number of $(j,k)$ with $j,k\in [i-1]$ and $u_iu_j,u_ju_k\in E(G)$ 
and $u_iu_k\not\in E(G)$.
\end{itemize}
Note that the set $\left[s_i^{(1)}+s_i^{(2)}\right]_0$ contains $1+s_i^{(1)}+s_i^{(2)}$ elements,
which is at least one more than the number of excluded valid choices.
Since 
$s_i^{(1)}\leq d_G(u_i)\leq \Delta$
and 
$$s_i^{(2)}\leq \sum\limits_{v\in N_G(u_i)}(d_G(v)-1)\leq \Delta(\Delta-1),$$
we have $s_i^{(1)}+s_i^{(2)}\leq \Delta^2$, which implies (\ref{e1}).

We are now in a position to show our second main result.

\begin{proof}[Proof of Theorem \ref{theorem1}]
Consider the application of the above greedy algorithm to a linear ordering $u_1,\ldots,u_n$ 
of the vertices of $G$ 
that is chosen uniformly at random.
For $i\in [n]$, let $d_i^-$ be the number of neighbors $u_j$ of $u_i$ with $j<i$, and let $d_i^+=d_G(u_i)-d_i^-$.
By the definitions of $s_i^{(1)}$ and $s_i^{(2)}$, double-counting implies
\begin{eqnarray*}
\sum\limits_{i=1}^ns_i^{(1)} &=& \sum\limits_{j=1}^nd_j^+\mbox{ and }\\
\sum\limits_{i=1}^ns_i^{(2)} &\leq & \sum\limits_{j=1}^n\left(d_j^-d_j^++{d_j^+\choose 2}\right).
\end{eqnarray*}
Using (\ref{e5}), linearity of expectation, and the uniform random choice of the linear ordering, we obtain
\begin{eqnarray*}
P^t(G) 
& \leq & \mathbb{E}\left[\sum\limits_{i=1}^n\left(s_i^{(1)}+s_i^{(2)}\right)\right]\\
& \leq & \mathbb{E}\left[\sum\limits_{j=1}^n\left(d_j^++ d_j^-d_j^++{d_j^+\choose 2}\right)\right]\\
& = & \sum\limits_{j=1}^n\mathbb{E}\left[d_j^++ d_j^-d_j^++{d_j^+\choose 2}\right]\\
& = & 
\sum\limits_{j=1}^n
\left(\frac{1}{d_G(u_j)+1}
\sum\limits_{d^-=0}^{d_G(u_j)}\left(\Big(d_G(u_j)-d^-\Big)+d^-\Big(d_G(u_j)-d^-\Big)+{d_G(u_j)-d^-\choose 2}\right)\right)\\
& = & 
\sum\limits_{j=1}^n\frac{1}{6}d_G(u_j)(2d_G(u_j)+1),
\end{eqnarray*}
which completes the proof.
\end{proof}
For a cubic graph $G$ of order $n$, Theorem \ref{theorem1} implies $P^t(G)\leq 3.5n$.
Executing the greedy algorithm on all cubic graphs $G$ of order $n\leq 20$
yields proper pushing schemes $\rho:V(G)\to\mathbb{N}_0$
where $\sum_{u\in V(G)}\rho(u)$ behaves roughly like $5n/8$ on average,
that is, much better than guaranteed by Theorem \ref{theorem1}.
Similarly, executing the greedy algorithm on all $4$-regular graphs $G$ 
of order $n\leq 15$
yields proper pushing schemes $\rho$
where $\sum_{u\in V(G)}\rho(u)$ behaves roughly like $2n/3$ on average.

The following two propositions point to possible improvements 
of (the analysis of) the greedy algorithm.
Recall that the girth of a graph is the length of a shortest cycle.

\begin{proposition}\label{proposition1}
If $G$ is a cubic graph of girth at least $5$, 
then the greedy algorithm applied to some linear ordering of the vertices of $G$ 
yields a proper pushing scheme $\rho:V(G)\to\mathbb{N}_0$ with 
$$\sum\limits_{u\in V(G)}\rho(u)\leq \left(3.5-23/840\right)n.$$
\end{proposition}
\begin{proof}
As in the proof of Theorem \ref{theorem1}, 
consider the application of the greedy algorithm to a linear ordering $u_1,\ldots,u_n$ 
of the vertices of $G$ 
that is chosen uniformly at random.
For $i\in [n]$, let $g_i$ be the final pushing value assigned to $u_i$ 
by the greedy algorithm, that is, $g_i=\rho_i(u_i)=\rho_n(u_i)$.
Theorem \ref{theorem1} relies on the estimate $g_i\leq s_i:=s_i^{(1)}+s_i^{(2)}$.
In order to quantify the improvement of this estimate, 
we define $t_1,\ldots,t_n$ as follows:
\begin{quote}
{\it Initialize $t_i$ as $0$ for every $i\in [n]$.
For $i$ from $1$ up to $n$ proceed as follows:
If $u_i$ has no neighbor $u_j$ with $j>i$, then increase $t_i$ by $s_i-g_i$.
Otherwise, let $u_j$ be the neighbor of $u_i$ with largest index $j$,
increase $t_i$ by $(s_i-g_i)/4$ and $t_j$ by $3(s_i-g_i)/4$.}
\end{quote}
Clearly, we have
$\sum\limits_{i=1}^n g_i
=\sum\limits_{i=1}^n s_i-\sum\limits_{i=1}^n t_i$.
By linearity of expectation, we have
\begin{eqnarray}\label{e5b}
\mathbb{E}\left[\sum\limits_{i=1}^n g_i\right]
&=&
\mathbb{E}\left[\sum\limits_{i=1}^n s_i\right]
-
\mathbb{E}\left[\sum\limits_{i=1}^n t_i\right]
\leq 3.5n-\sum\limits_{i=1}^n \mathbb{E}\left[t_i\right].
\end{eqnarray}
The assumptions that $G$ is cubic and of girth at least $5$
simplify the estimation of $\mathbb{E} \left[ t_i\right]$.

Let $i\in[n]$.
Let $u=u_i$,
let $N_G(u)=\{ v_1,v_2,v_3\}$, and, for $i\in [3]$, 
let $N_G(v_i)=\{ u,v_{i,1},v_{i,2}\}$.
Since $G$ has girth at least $5$, the set $U=\{ u,v_1,v_2,v_3,v_{1,1},v_{1,2},v_{2,1},v_{2,2},v_{3,1},v_{3,2}\}$
contains exactly $10$ distinct vertices.
Restricting the linear ordering $u_1,\ldots,u_n$ to $U$ yields a permutation $\pi$ from $S_U$
and, since the linear ordering is chosen uniformly at random,
each of the $10!$ permutations in $S_U$ is equally likely to be the restriction.
Note that $s_i$ is completely determined by $\pi$.

There is a set $S_1$ of exactly $\frac{3\cdot 2!\cdot 2!}{6!}\cdot 10!$ permutations in $S_U$,
corresponding to linear orderings of $U$, in which 
\begin{itemize}
\item exactly one neighbor $v$ of $u$ comes before $u$ and 
\item the two neighbors of $v$ that are distinct from $u$ both come before $v$. 
\end{itemize}
There is a set $S_2$ of exactly $\frac{3\cdot 2!\cdot 2!}{6!}\cdot 10!$ permutations in $S_U$ 
in which 
\begin{itemize}
\item exactly one neighbor $v$ of $u$ comes before $u$,
\item exactly one of the two neighbors of $v$ that are distinct from $u$ comes before $v$, and
\item the other neighbor of $v$ that is distinct from $u$ comes after $v$ and before $u$. 
\end{itemize}
There is a set $S_3$ of exactly $3!\cdot 6!$ permutations in $S_U$ 
in which 
\begin{itemize}
\item all three neighbors $v_1$, $v_2$, and $v_3$ of $u$ come before $u$ and 
\item all six vertices $v_{1,1}$, $v_{1,2}$, $v_{2,1}$, $v_{2,2}$, $v_{3,1}$, and $v_{3,2}$ come after $u$. 
\end{itemize}
If $\pi\in S_U\setminus \left(S_1\cup S_2\cup S_3\right)$, then $t_i\geq 0$.
Now, let $\pi\in S_1$.
Let $u_j$ be the unique neighbor of $u_i$ with $j<i$.
Note that $s_i=3$ and $s_j\geq 2$.
If $\sigma_{i-1}(u_j)>3$, then $g_i=0$, which implies $t_i\geq (s_i-g_i)/4=3/4$.
If $\sigma_{i-1}(u_j)=3$, then $g_j=0$, which implies $t_i\geq 3(s_j-g_j)/4\geq 3/2$.
Altogether, we obtain $t_i\geq 3/4$.
Next, let $\pi\in S_2$.
Let $u_j$ be the unique neighbor of $u_i$ with $j<i$.
Note that $s_i=3$ and $s_j\geq 1$.
If $\sigma_{i-1}(u_j)>3$, then $g_i=0$, which implies $t_i\geq (s_i-g_i)/4=3/4$.
If $\sigma_{i-1}(u_j)=3$, then $g_j=0$, which implies $t_i\geq 3(s_j-g_j)/4\geq 3/4$.
Altogether, we obtain $t_i\geq 3/4$.
Finally, let $\pi\in S_3$.
In this case $s_i=3$ and $g_i=1$,
which implies $t_i\geq s_i-g_i=2$.

By the uniform random choice of the linear ordering, we obtain
\begin{eqnarray*}
\mathbb{E} \left[ t_i\right] 
& \geq & \frac{1}{10!}\left(\frac{3}{4}|S_1|+\frac{3}{4}|S_2|+2|S_3|\right)\\
&\geq &\frac{1}{10!}\left(\frac{3}{4}\cdot \frac{3\cdot 2!\cdot 2!}{6!}\cdot 10!+\frac{3}{4}\cdot \frac{3\cdot 2!\cdot 2!}{6!}\cdot 10!+2\cdot 3!\cdot 6!\right)\\
&=& 23/840,
\end{eqnarray*}
and (\ref{e5b}) completes the proof.
\end{proof}
It is possible to generalize
Proposition \ref{proposition1} to larger regularities 
and also to graphs of any girth.

Provided that $G$ is nice, connected, 
and of maximum degree at most $\Delta$ but not $\Delta$-regular,
Bensmail et al.~\cite{bemaor} use a breadth-first search argument 
to show that $P^1(G)\leq \Delta^2-\Delta$.
Our final result also builds on the greedy algorithm.

\begin{proposition}\label{proposition2}
If $G$ is a nice graph of maximum degree at most $\Delta$, then $P^1(G)\leq \Delta^2-1$.
\end{proposition}
\begin{proof}
Clearly, we may assume that $G$ is connected and $\Delta$-regular.
Since $P^1(K_{\Delta,\Delta})=1$ (cf. Theorem 3.1 in \cite{bemaor}),
we may assume that $G$ is not $K_{\Delta,\Delta}$.
Let $u_1,\ldots,u_n$ be a reverse breadth-first search ordering, that is,
in particular, for every $i\in [n-1]$, the vertex $u_i$ has a neighbor $u_j$ with $j>i$.
This implies that the function $\rho_{n-1}:V(G)\to \mathbb{N}_0$ 
that is produced by the greedy algorithm 
and satisfies $P_{n-1}^{(1)}$ and $P_{n-1}^{(2)}$
only uses values in $\left[\Delta^2-\Delta\right]_0$.
In particular, only the final vertex $u_n$ may require $\Delta^2$ pushes.

Let $u=u_n$,
let $N_G(u)=\{ v_1,\ldots,v_{\Delta}\}$, and, for every $i\in [\Delta]$,
let $N_G(v_i)=\{ u,v_{i,1},\ldots,v_{i,\Delta-1}\}$.
Since $G$ is not $K_{\Delta,\Delta}$,
we may assume that $v_{1,1}$ is not adjacent to all vertices in $N_G(u)$.
In view of the desired statement, 
we may assume that $\Delta^2$ is the only valid choice for $\rho_n(u_n)$
within the set $[\Delta^2]_0$.
More presicely, for every $s$ in $[\Delta^2-1]_0$,
there is an edge $e(s)$ incident with a neighbor of $u$, say $e(s)=xy$,
such that setting $\rho_n(u)$ to $s$ results in $\sigma_n(x)=\sigma_n(y)$.
Note that there are at most $\Delta^2$ edges incident with neighbors of $u$.
It follows that $u$ does not lie in a triangle,
that the edge $e(s)$ is unique for every $s\in [\Delta^2-1]_0$, and 
that $\{ e(s):s\in [\Delta^2-1]_0\}$ is the set of all $\Delta^2$ distinct edges 
that are incident with neighbors of $u$.
For $i\in [\Delta]$ and $j\in [\Delta-1]$, 
let $s_i$ be such that $e(s_i)=uv_i$
and 
let $s_{i,j}$ be such that $e(s_{i,j})=v_iv_{i,j}$.

The set $[\Delta^2-\Delta]_0\cup \{ \Delta^2-\Delta+1\}$ 
contains a value $t'$ distinct from $t=\rho_{n-1}(v_1)$ 
such that changing within $\rho_{n-1}$ the value of $v_1$ from $t$ to $t'$
yields a function $\rho'_{n-1}$ 
that satisfies $P_{n-1}^{(1)}$ and $P_{n-1}^{(2)}$.
In view of the desired statement, 
we may again assume that, for every $s'$ in $[\Delta^2-1]_0$,
there is an edge $e'(s')$ incident with a neighbor of $u$, say $e'(s')=xy$,
such that changing within $\rho'_{n-1}$ 
the value of $u$ from $0$ to $s'$
yields a function $\rho'_n$ with $\sigma'_n(x)=\sigma'_n(y)$,
where $\sigma'_n$ is derived from $\rho'_n$ in the obvious way.
As before,
the edge $e'(s')$ is unique for every $s'\in [\Delta^2-1]_0$ and 
$\{ e'(s'):s'\in [\Delta^2-1]_0\}$ is the set of all $\Delta^2$ distinct edges 
that are incident with neighbors of $u$.
For $i\in [\Delta]$ and $j\in [\Delta-1]$, 
let $s'_i$ be such that $e'(s'_i)=uv_i$
and 
let $s'_{i,j}$ be such that $e'(s'_{i,j})=v_iv_{i,j}$.

Note that 
\begin{eqnarray}
[\Delta^2-1]_0 &=& 
\{ s_i':i\in [\Delta]\}\cup \{ s'_{i,j}:i\in [\Delta]\mbox{ and }j\in [\Delta-1]\}\nonumber\\
&=&\{ s_i:i\in [\Delta]\}\cup \{ s_{i,j}:i\in [\Delta]\mbox{ and }j\in [\Delta-1]\}.\label{e6}
\end{eqnarray}
Since $\rho'_{n-1}$ and $\rho_{n-1}$ only differ at $v_1$,
if follows, for $\alpha=t'-t=\rho'_{n-1}(v_1)-\rho_{n-1}(v_1)$, that 
\begin{itemize}
\item $s'_1=s_1+\alpha$,
\item $s_i'=s_i-\frac{\alpha}{\Delta-1}$ for every $i\in [\Delta]\setminus \{ 1\}$, and
\item $s'_{1,j}=s_{1,j}-(\Delta-1)\alpha$ for every $j\in [\Delta-1]$.
\end{itemize}
Now, let $i\in [\Delta]\setminus \{ 1\}$ and $j\in [\Delta-1]$.
If 
$v_{i,j}$ is not adjacent to $v_1$, then $s'_{i,j}=s_{i,j}$,
and, if 
$v_{i,j}$ is adjacent to $v_1$, then $s'_{i,j}=s_{i,j}+\alpha$.
Note that 
$\sum\limits_{i=1}^{\Delta}(s_i'-s_i)=\alpha-\frac{\alpha}{\Delta-1}(\Delta-1)=0$.
Furthermore, if $v_{1,\ell}$ has exactly $k$ neighbors in $N_G(u)$, then 
the sum of $s_{i,j}'-s_{i,j}$
over all edges $v_iv_{i,j}$ with $v_{i,j}=v_{1,\ell}$ 
equals 
$-(\Delta-1)\alpha+(k-1)\alpha=-(\Delta-k)\alpha$.
Since $v_{1,1}$ is not adjacent to all vertices in $N_G(u)$, it follows that 
$$\sum\limits_{i=1}^\Delta \left(s_i'+\sum\limits_{j=1}^{\Delta-1}s'_{i,j}\right)
\not=\sum\limits_{i=1}^\Delta \left(s_i+\sum\limits_{j=1}^{\Delta-1}s_{i,j}\right).$$
This contradicts (\ref{e6}), which completes the proof.
\end{proof}
There is a weakness in the analysis of the greedy algorithm 
that we have not been able to overcome: As explained before (\ref{e5}),
at most $s_i^{(1)}+s_i^{(2)}$ of the $s_i^{(1)}+s_i^{(2)}+1$ elements 
in $\left[s_i^{(1)}+s_i^{(2)}\right]_0$ are excluded by the required conditions.
Now, the analysis of the greedy algorithm 
pessimistically assumes that the smallest non-excluded
element of that set is its maximum, which seems unlikely as least on average.

\end{document}